\documentclass{elsarticle}
\usepackage{amsmath,amssymb}
\usepackage{amsthm}
\usepackage{bbm}
\theoremstyle{plain}
\newtheorem{theorem} {Theorem}
\newtheorem{corollary}{Corollary}
\newtheorem{definition} {Definition}
\theoremstyle{remark}
\newtheorem{remark}{Remark}
\newtheorem{example}{Example}
\newtheorem{lemma}{Lemma}
\def\be{\begin{equation}}
\def\ee{\end{equation}}
\newcommand{\R}{\mathbb{R}}
\newcommand{\pr}{\mathsf P}
\newcommand{\F}{\mathcal{F}}
\newcommand{\Ll}{\mathcal{L}}

\newcommand{\set}[1]{\left\{#1\right\}}
\newcommand{\ex}[1]{\mathsf{E}\left[\,#1\,\right]}
\newcommand{\abs}[1]{\left\vert#1\right\vert}
\journal{Journal of Differential Equations}
\begin{document}
\title{On the distribution of local times and integral functionals of a homogeneous diffusion process}

\author[m1]{Mykola Perestyuk}
\ead{pmo@univ.kiev.ua}

\author[m1]{Yuliya Mishura}
\ead{myus@univ.kiev.ua}

\author[m1]{Georgiy Shevchenko\corref{c1}}

\cortext[c1]{Corresponding author}
\ead{zhora@univ.kiev.ua}

\address[m1]{Kiev National Taras Shevchenko University,
Department of Mechanics and Mathematics,
Volodomirska 60,
01601 Kiev,
Ukraine}

\begin{abstract}
In this article we study a homogeneous transient diffusion process $X$. We combine the theories of differential equations and of stochastic processes to obtain new results for homogeneous diffusion processes, generalizing the results of Salminen and Yor. The distribution of local time of $X$ is found in a closed form. To this end, a second order differential equation corresponding to the generator of $X$ is considered, and properties of its monotone solutions as functions of a parameter are established using their probabilistic representations. We also provide expressions and upper bounds for moments, exponential moments, and potentials of integral functionals of $X$.
\end{abstract}


\maketitle

\section{Introduction}

We consider a family $\{X_t^x,t\ge0,x\in\R\}$ of one-dimensional homogeneous diffusion processes
 defined on a complete filtered probability space  $\{\Omega,\F,\{\F_t\}_{t\ge0},\pr\}$ by a stochastic differential equation
\begin{equation*}
dX_t^x=b(X_t^x)dt+a(X_t^x)dW_t,\quad t\ge0,
\end{equation*}
with the initial condition $X_0^x=x\in\R$; here $\{W_t,t\ge0\}$ is a standard $\F_t$-Wiener process. 
If the initial condition is not important, we will denote the process in question by $X$. 
Let the coefficients $a,b$ of equation  (\ref{f1})
be continuous and satisfy any conditions of the existence of a non-explosive weak solution on $\R$. Assume also  $a(x)\neq0 $ for  $x\in\R$. Let also $a(x)\neq0,$ $x\in\R$. Below we introduce several objects related to the family 
$\{X_t^x,t\ge0,x\in\R\}$.

The generator of diffusion process $X$ is defined for $f\in C^2(\R)$ as $$\Ll f(x)=\frac{a(x)^2}{2}f''(x)+b(x)f'(x).$$
Define the functions
$$\varphi(x_0,x) = \exp\set{-2\int_{x_0}^z \frac{b(u)}{a(u)^2}du},\quad \Phi(x_0,x)=\int_{x_0}^x
\varphi(x_0,z)dz,\quad x_0,x\in\R\cup\{-\infty,+\infty\}.$$
It is easy to see that for a fixed $x_0\in\R$ the function $\Phi(x_0,\cdot)$ solves a second order homogeneous differential equation $\Ll\Phi(x_0,\cdot)=0$.

For $x,y\in\R$, let $\tau_y^x=\inf\{t\ge 0, X_t^x=y\}$  be the first moment of hitting point $y$. For any $(a,b)\subset \R$ and $x\in(a,b)$, let $\tau_{a,b}^x=\inf\{t\ge 0, X_t^x\notin (a,b)\} = \tau_a^x \wedge \tau_b^x$ be the first moment of exiting interval $(a,b)$. (We use the convention $\inf \varnothing = +\infty$.)

For any $t>0$ and $y\in\R$, define  
a local time of process  $X^x$ at the point $y$ on the interval $[0,t]$ by
\begin{equation}\label{local time}
L^x_t(y)=a(y)^2\lim_{\varepsilon\downarrow0}\frac{1}{2\varepsilon}
\int_0^t \mathbb{I}\{|X_s^x-y|\leq\varepsilon\}ds.
\end{equation} 
(The factor
$a(y)^2$ is included to agree with the general Meyer--Tanaka definition of a local time of a semimartingales \cite{b1}.)
The limit in \eqref{local time} exists almost surely and defines a continuous non-decreasing process
$\set{L_t^x(y),t\ge0}$ for any $x,y\in\R$. The local time on the whole interval $[0,+\infty)$ will be denoted by
$L^x_{\infty}(y)=\lim_{t\to+\infty}L_t^x(y)$.

In this article we focus on the \textit{transient} diffusion processes, i.e.\ those converging to $+\infty$ or $-\infty$ as $t\to-\infty$. First we determine the probabilistic distribution of
$L_\infty^x$ explicitly, in terms of the coefficients $a$ and $b$. This problem was considered in papers \cite{b1} and \cite{tour}, however, the distribution parameter were not determined explicitly, but  rather through limits of some functionals of solutions to inhomogeneous differential equations, see \eqref{f2}. Second, we use the distribution of $L^x_\infty$ to study integral functionals of the form $J_\infty(f)=\int_0^\infty f(X_s^x)ds$, which can be interpreted as continuous perpetuities in the framework of financial mathematics. We follow the approach of Salminen and Yor, used in \cite{SalmYor} to study integral functionals of a Wiener process with a positive drift, and generalize their results to homogeneous transient diffusion processes. Applying the results of \cite{mijauru}, we establish criteria of convergence  of almost sure finiteness of the functionals $J_\infty(f)$, calculate their  moments and potentials and bound their exponential moments.

\section{The distribution of a local time of a transient diffusion process}
According to the classical results (see e.g.\ \cite{b3}), in the case where 
$\Phi(x,+\infty)=-\Phi(x,-\infty)=+\infty$ for some (equivalently, for all)
$x\in\R$, the diffusion process  $X$ is recurrent, i.e.
$$\pr\left(\limsup_{t\to+\infty}X^x_t=+\infty,\liminf_{t\to+\infty}
X^x_t=-\infty\right)=1.$$ Therefore, $L_{\infty}^x(y)=+\infty$ for all
$x,y\in\R$ a.s. The behavior of the inverse local time process was studied in the recurrent case in  \cite{b1,salvalyor,salval}.

For this reason, in what follows we will consider only the case of a transient process $X$, where at least one 
of the integrals $\Phi(x_0,+\infty)$ and $\Phi(-\infty,x_0)$ is finite.

It is sufficient to consider the case $x=y$ only. Indeed, by the strong Markov property of the process  $X$, for any  $l\ge 0$,
$$
\pr(L^x_{\infty} (y)>l) = \pr(L^y_{\infty}(y)>l)\pr(\tau_y^x<+\infty).
$$
The probability  $\pr(\tau_y^x<+\infty)=1-\pr(\tau_y^x=+\infty)$ can be found with the help of еру well-known formula (see e.g.\ \cite[Section VIII.6, (18)]{b5}): for $x\in(a,b)$
\begin{equation*}
\pr(X^x_{\tau_{a,b}} = b) = \frac{\Phi(a,x)}{\Phi(a,b)}.
\end{equation*}
Then the value of probability in question depends on  $x$, $y$, and integrals $\Phi(x,+\infty)$, $\Phi(x,-\infty)$. Specifically, if $x>y$, then
$$
\pr(\tau_y^x=\infty) = \lim_{a\to+\infty} \pr(X^x_{\tau_{y,a}}=a) = 
 \lim_{a\to+\infty} \frac{\Phi(y,x)}{\Phi(y,a)},
$$
whence
\begin{equation}\label{Ptaux<y}
\pr(\tau_y^x=+\infty) = \begin{cases}
\frac{\Phi(y,x)}{\Phi(y,+\infty)}, & \Phi(x,+\infty)<+\infty,\\
0,& \Phi(x,+\infty)=+\infty.
\end{cases}
\end{equation}
For  $x<y$
\begin{gather*}
\pr(\tau_y^x=\infty) = \lim_{a\to-\infty} \left(1-\pr(X^x_{\tau_{a,y}}=y)\right) = 
 \lim_{a\to-\infty} \frac{\Phi(a,y)-\Phi(a,x)}{\Phi(a,y)} = \\=
 \lim_{a\to-\infty} \frac{\phi(a,x)\Phi(x,y)}{-\phi(a,y)\Phi(y,a)} =
 \lim_{a\to-\infty} \frac{-\phi(a,x)\phi(x,y)\Phi(y,x)}{-\phi(a,y)\Phi(y,a)} = \lim_{a\to-\infty} \frac{\Phi(y,x)}{\Phi(y,a)};
\end{gather*}
therefore
\begin{equation}\label{Ptaux>y}
\pr(\tau_y^x=+\infty) = \begin{cases}
\frac{\Phi(y,x)}{\Phi(y,-\infty)}, & -\Phi(x,-\infty)<+\infty,\\
0,& -\Phi(x,-\infty)=+\infty.
\end{cases}
\end{equation}
Thus it is indeed sufficient to determine distributions of variables $L^x_\infty(x)$. To this end, we will use the following facts.

\begin{itemize}
\item[(i)] According to   \cite[Theorem 1]{b1}, $\pr(L^x_\infty(x)>l)=\exp(-l\psi_x(0)),$ where
\be\label{f2}
\psi_x(0)=\psi^{x,+}(0)+\psi^{x,-}(0),\qquad
\psi^{x,\pm}(0)=\pm\frac{1}{2}\lim_{\lambda\downarrow0}\frac{y'_{\lambda,\pm}(x)}{y_{\lambda,\pm}(x)};
\ee
here for a fixed $\lambda>0$ the functions $y_{\lambda,+}$ та $y_{\lambda,-}$ are, respectively, the increasing and decreasing solutions of the equation
\be\label{f3}\Ll y=\lambda y.
\ee

\item[(ii)]According to \cite{b3}, the functions $y_{\lambda,+}$ and $y_{\lambda,-}$ admit the probabilistic representations
\be\label{f4} y_{\lambda,+}(x)=\begin{cases}E e^{-\lambda \tau_0^x},&x<0,\\(E e^{-\lambda \tau_x^0})^{-1},&x\ge0\end{cases}
\qquad y_{\lambda,-}(x)=\begin{cases}E e^{-\lambda \tau_0^x},&x\ge0,\\(E e^{-\lambda \tau_x^0})^{-1},&x<0.\end{cases}
\ee
(We set $e^{-\lambda t}=0$ for $t=+\infty$, $\lambda>0$.)

\item[(iii)] Any solution $y_{\lambda}(x)$ to equation  \eqref{f3}
has the integral representation
\be\label{f6}y_\lambda(x)=C_1(\lambda)+C_2(\lambda)\Phi(x_0,x)+2\lambda\int_{x_0}^x
\frac{\Phi(s,x)}{a^2(s)}y_\lambda(s)ds,
\ee

\be\label{f7}y'_\lambda(x)=C_2(\lambda)\varphi(x_0,x)+2\lambda\int_{x_0}^x\frac{\varphi(s,x)}{a(s)^2}y_\lambda(s)ds.
\ee
\end{itemize}
With this facts at hand, we are able to prove the main result of this section.
\begin{theorem}\label{thm:psix} The value of the parameter $\psi_x(0)$ can be written explicitly as
\begin{equation}\label{psix}
\psi_x(0) = \frac{1}2 \left( 
\frac{1}{\Phi(x,+\infty)} - \frac{1}{\Phi(x,-\infty)}\right),
\end{equation} 
where $\frac{1}{\infty}:=0$.
\end{theorem}
\begin{proof}
We need to identify the limits in  equalities  (\ref{f2}). To this end, we start by writing   the representation (\ref{f6}) for
$y_{\lambda, +}(x)$ and putting $x_{0}=0$. Since it follows from  (\ref{f4}) that $y_{\lambda, +}(0)=1$, we can deduce from  (\ref{f6}) and  (\ref{f7}) that
\begin{equation}\label{f8}
y_{\lambda,
+}(x)=1+c_{2}(\lambda)\Phi(0,x)+2\lambda\int_{0}^x
\frac{\Phi(s,x)}{a(s)^2}y_{\lambda,+}(s)ds
\end{equation}
and  \begin{equation}\label{f9}
y_{\lambda,
+}'(x)=c_{2}(\lambda)\varphi(0,x)+2\lambda\int_{0}^x
\frac{\varphi(s,x)}{a(s)^2}y_{\lambda,+}(s)ds.
\end{equation}
Now plug $x=0$ into (\ref{f9}). Then it follows from the equality $\phi(0,0)=1$ that
$y_{\lambda, +}'(0)=c_{2}(\lambda)$. Therefore,  (\ref{f8}) and
(\ref{f9}) can be transformed  to
\begin{equation}\label{f10}
y_{\lambda, +}(x)=1+y_{\lambda,
+}'(0)\Phi(0,x)+2\lambda\int_{0}^x
\frac{\Phi(s,x)}{a(s)^2}y_{\lambda,+}(s)ds
\end{equation}
and  \begin{equation}\label{f11}
y_{\lambda, +}'(x)=y_{\lambda,+}'(0)\varphi(0,x)+2\lambda\int_{0}^{x}
\frac{\varphi(s,x)}{a(s)^2}y_{\lambda,+}(s)ds.
\end{equation}
Now, for a fixed $x\in\mathbb{R}$, we need to proceed to a limit as  $\lambda \rightarrow 0$ in  equations
(\ref{f10}) and  (\ref{f11}).
Assume that $x>0$. Then
the integrands in (\ref{f10}) and  (\ref{f11}) are positive,  furthermore,  the functions
${\Phi(s,x)}/{a(s)^2}$ and
${\varphi(s,x)}/{a(s)^2}$ are bounded. Clearly, we can assume that $\lambda\in(0,1]$. It is evident from
(\ref{f4}) that  $y_{\lambda, +}(x)$ is increasing in
$\lambda$ for $x>0$. It is also increasing in $x$, therefore, $0\leq y_{\lambda, +}(s)\leq
(E_{0}e^{-\tau_{x}^{0}})^{-1} $. Consequently,
$$\lim _{\lambda \downarrow 0} \lambda \int_0^x
\frac{\Phi(s,x)}{a(s)^2}y_{\lambda,
+}(s)ds=\lim _{\lambda \downarrow 0} \lambda
\int_0^x
\frac{\varphi(s,x)}{a(s)^2}y_{\lambda,
+}(s)ds=0.$$ The same conclusion holds  for  $x<0$.   In this case  we change the sign
of   $\Phi(s,x)$ and note that $0\leq
y_{\lambda, +}(s)\leq 1$ for  $s<0$.  Now we can find  $\lim
_{\lambda \downarrow 0} y_{\lambda, +}'(0)$. Note at first that it follows from the  representations
(\ref{f4})
that  $y_{\lambda, +}(x)$ and
$y_{\lambda, -}(x)$ are continuous in  $\lambda$ for any $x \in \mathbb{R}$.  Therefore the left-hand side of
(\ref{f10}) and the integral in its right-hand side are continuous in
$\lambda$. It means that  $y_{\lambda, +}'(0)$ and  $(y_{\lambda,
-})'(0)$ are continuous in  $\lambda$ as well. So, the limits $\lim
_{\lambda \downarrow 0} y_{\lambda, \pm}'(0)$ are equal to the values of the corresponding derivatives at zero:
$$\lim _{\lambda \downarrow 0} y_{\lambda, \pm}'(0)=\lim
_{x \rightarrow 0} \left(\frac{y_{\lambda,
\pm}(x)-1}{x}\Big|_{\lambda=0}\right).$$ 

Further, for any $\lambda>0$ we have the equalities
$E(e^{-\lambda
\tau_{0}^{x}}-1)=-P\{\tau^{x}_{0}=+\infty\}+E(e^{-\lambda
\tau_{0}^{x}}-1) \mathbb{I} \{\tau^{x}_{0}<+\infty\}$, therefore,
$E(e^{-\lambda
\tau_{0}^{x}}-1)\big|_{\lambda=0}=-P\{\tau^{x}_{0}=+\infty\}$.
We can conclude that
\begin{gather*}
\lim _{\lambda \downarrow 0} y_{\lambda,
+}'(0)=-\lim _{x \uparrow
0}\frac{P\{\tau_{0}^{x}=+\infty\}}{x}.
\end{gather*}
similarly,
\begin{gather*}
\lim _{\lambda
\downarrow 0} y_{\lambda, -}'(0)=-\lim _{x \downarrow
0}\frac{P\{\tau_{0}^{x}=+\infty\}}{x}.
\end{gather*}
From \eqref{Ptaux<y}--\eqref{Ptaux>y} we have that
\begin{equation*}
\pr(\tau_0^x=+\infty) = \begin{cases}
\frac{\Phi(0,x)}{\Phi(0,-\infty)}, & -\Phi(0,-\infty)<+\infty,\\
0,& -\Phi(0,-\infty)=+\infty
\end{cases}
\end{equation*}
for $x<0$, and
\begin{equation*}
\pr(\tau_0^x=+\infty) = \begin{cases}
\frac{\Phi(0,x)}{\Phi(0,+\infty)}, & \Phi(0,+\infty)<+\infty,\\
0,& \Phi(0,+\infty)=+\infty.
\end{cases}
\end{equation*}
for  $x>0$. Note that $$\frac{\Phi(0,x)}{x} = \frac{\Phi(0,x)-\Phi(0,0)}{x-0}\to
 \Phi'_x(0,x)|_{x=0} = \varphi(0,0)=1,\quad x\to 0.$$ Hence, we obtain
\begin{equation*} 
\lim _{\lambda \downarrow 0} y_{\lambda,
+}'(0)=-\frac{1}{\Phi(0,-\infty)},\quad
\lim _{\lambda \downarrow 0} y_{\lambda,
-}'(0)= -\frac{1}{\Phi(0,+\infty)}.
\end{equation*}
Substitute these limits into (\ref{f10})--(\ref{f11}), we obtain from the properties of the functions  $\Phi$ and  $\varphi$ that
\begin{gather*}
\lim _{\lambda \downarrow 0} y_{\lambda, +}(x) =
\begin{cases}
\frac{\phi(0,x)\Phi(x,-\infty)}{\Phi(0,-\infty)}, & -\Phi(0,-\infty)<+\infty,\\
1, & -\Phi(0,-\infty)=+\infty;
\end{cases}
\\
\lim _{\lambda \downarrow 0} y_{\lambda, -}(x) =
\begin{cases}
\frac{\phi(0,x)\Phi(x,+\infty)}{\Phi(0, +\infty)}, & \Phi(0,+\infty)<+\infty, \\
1, & \Phi(0,+\infty)=+\infty;
\end{cases}\\
\lim _{\lambda \downarrow 0} y_{\lambda, +}'(x)  = -\frac{\phi(0,x)}{\Phi(0,-\infty)},\quad
\lim _{\lambda \downarrow 0} y_{\lambda,
-}'(x)= -\frac{\phi(0,x)}{\Phi(0,+\infty)}.
\end{gather*}
Now the proof follows from \eqref{f2}.
\end{proof}

\begin{corollary}\label{cor1}
1. In each of the cases: $x=y$;  $x<y$ and $-\Phi(0,-\infty) = +\infty$; $x>y$ and $\Phi(0,+\infty) = +\infty$, the local time  $L_\infty^x(y)$ is exponentially distributed with parameter $\psi_y(0)$ given by \eqref{psix}.

2. If $x<y$ and $-\Phi(0,-\infty) < +\infty$, then the local time $L_\infty^x(y)$ is distributed as $\kappa \xi$, where $\xi$ is exponentially distributed with parameter $\psi_y(0)$,  $\kappa$ is an independent of $\xi$ Bernoulli random variable with $$\pr(\kappa =0) =1-\pr(\kappa=1) = \frac{\Phi(y,x)}{\Phi(y,-\infty)}.$$

3. If $x>y$ and $\Phi(0,+\infty) < +\infty$, then the local time $L_\infty^x(y)$ is distributed as $\kappa \xi$, where $\xi$ is exponentially distributed with parameter $\psi_y(0)$,  $\kappa$ is an independent of $\xi$ Bernoulli random variable with $$\pr(\kappa =0) =1-\pr(\kappa=1) = \frac{\Phi(y,x)}{\Phi(y,+\infty)}.$$
\end{corollary}


\begin{example}
Let $a(x)\equiv a\neq 0$ and  $b(x)\equiv b$ be constant. Then  $\varphi(x,y) = e^{-2 b (y-x)/a^2}$,
$\Phi(x,y) = \frac{a^2}{2b}(1-e^{-2b(y-x)/a^2})$ for  $b\neq 0$ and  $\Phi(x,y) = y-x$ for $b=0$. In this case diffusion process $X$ is transient if and only if $b\neq 0$, moreover, $-\Phi(0,-\infty)=+\infty$ \  and $\Phi(0,+\infty)<+\infty$ for $b>0$, and   $-\Phi(0,-\infty)<+\infty$, $\Phi(0,+\infty)=+\infty$ for  $b<0$. The cases are symmetric, therefore we will consider only  the case $b>0$.

Equation \eqref{f3}  is a linear equation with constant coefficients, so its general solution has a form $y(x) = C_1 \exp\set{\frac{-b + \sqrt{b^2+2a^2\lambda}}{a^2}x}+C_2\exp\set{\frac{-b - \sqrt{b^2+2a^2\lambda}}{a^2}x}$. The increasing and decreasing solutions are $y_{\lambda,+}(x) = \exp\set{\frac{-b + \sqrt{b^2+2a^2\lambda}}{a^2}x}$ and  $y_{\lambda,-}(x) = \exp\set{\frac{-b - \sqrt{b^2+2a^2\lambda}}{a^2}x}$, respectively. Therefore,
$$
\frac{y_{\lambda,+}'(x)}{y_{\lambda,+}(x)} = \frac{-b + \sqrt{b^2+2a^2\lambda}}{a^2}\to 0, \lambda\downarrow 0;\quad
\frac{y_{\lambda,-}'(x)}{y_{\lambda,-}(x)} = \frac{-b - \sqrt{b^2+2a^2\lambda}}{a^2}\to -\frac{2b}{a^2}, \lambda\downarrow 0.
$$
Hence, $\psi_x(0) = \frac{b}{a^2}$ which coincides with the result of Theorem~\ref{thm:psix}, since in this case
$$\psi_x(0) =  \frac{1}{2 \Phi(x,+\infty)} = \frac{1}{2 {a^2}/{2b}} = \frac{b}{a^2}.
$$
Thus, for  $x\le y$ the local time $L^x_\infty(y)$ is exponentially distributed with a parameter $\frac{b}{a^2}$. For  $x>y$ the local time is distributed as  $\kappa \xi$, where $\xi$ has an exponential distribution with a parameter  $\frac{b}{a^2}$ and   $\kappa$ is Bernoulli random variable independent of   $\xi$ and distributed as $\pr(\kappa=1) =1-\pr(\kappa =0)= e^{-2b(x-y)/a^2}.$
Using the properties of exponential distribution, we see that these cases can be combined: $L^x_\infty(y)\overset{d}{=} \left(\xi-2(x-y)_+\right)_+ $, where $a_+ = a\vee 0$.
\end{example}
\begin{example}
Let $a(x)= \sqrt{x^2+1}$ and  $b(x)=x$. Then $\varphi(x,y) = \frac{x^2+1}{y^2+1}$ and
$\Phi(x,y) = (1+x^2)(\arctan y - \arctan x)$. We see that the process is transient and  $-\Phi(0,-\infty)=\Phi(0,\infty)=\frac\pi2<\infty$.

Due to Corollary \ref{cor1}, the local time $L^x_\infty(y)$ is distributed as $\kappa \xi$, where $\xi$ has an exponential distribution with a parameter
\begin{gather*}
\psi_x(0) =  \frac{1}{2 \Phi(x,+\infty)}-\frac{1}{2 \Phi(x,+\infty)} = \frac{1}{(1+x^2)(\pi - 2\arctan x)} - \frac{1}{(1+x^2)(\pi + 2\arctan x)}= \\ =\frac{4\arctan x}{(1+x^2)(\pi^2 - 4\arctan^2 x)}
\end{gather*}
and $\kappa$ is Bernoulli random variable, which is independent of   $\xi$ and distributed as $$\pr(\kappa=1) =1-\pr(\kappa =0)= \begin{cases}
\frac{\pi-2\arctan x}{\pi-2\arctan y}, & x\ge y,\\
\frac{\pi+2\arctan x}{\pi+2\arctan y}, & x< y.
\end{cases}
$$
\end{example}

\section{Integral functionals of a transient diffusion processes}

For a measurable function $f:\R\rightarrow \R$ such that $f$ and ${f}/{a^2}$ are locally integrable, define the integral functional $$J^x_\infty(f)=\int_0^\infty f(X_s^x)ds.$$
We will study the questions of finiteness and existence of moments of $J^x_\infty(f)$. We start with  the well-known occupation density formula (see e.g.\ \cite{b3, b1}):
\begin{equation}\label{eq1.17}
J^x_\infty(f)=\int_{\R}\frac{f(y)}{a(y)^2}L_{\infty}^x(y)dy.
\end{equation}
If the process $X^x$ is recurrent, then $L_{\infty}^x(y)=\infty$ a.s. for all $y\in\R$, so $J^x_\infty(f)$ is undefined unless $f$ is identically zero. Therefore, we will require that the process $X$ is transient. We remind that this holds iff $\Phi(0,+\infty)$ of $\Phi(0,-\infty)$ is finite. Moreover, if $\Phi(0,+\infty)< +\infty, \Phi(0,-\infty)=-\infty$, then $X_s^x\rightarrow +\infty$ a.s.; if  $\Phi(0,+\infty)= +\infty, \Phi(0,-\infty)>-\infty$ then $X_s^x\rightarrow -\infty$ a.s.; if  $\Phi(0,+\infty)< +\infty, \Phi(0,-\infty)>-\infty$ then $X_s^x\rightarrow +\infty$ on a set $A_+$ of positive probability,  and $X_s^x\rightarrow -\infty$ on a set $A_-=\Omega\setminus A_+$ of positive probability.

We start with a criterion of almost sure finiteness of $J^x_\infty (f)$. It was obtained in \cite{salmyourkhosh} in the case where only one of the integrals $\Phi(0,+\infty)$ of $\Phi(0,-\infty)$ is finite; the complete analysis was made in \cite{mijauru}.
Define 
$$
I_1(f)=\int_0^{+\infty} \frac{|f(y)|}{a(y)^2}\Phi(y,+\infty)dy,\quad I_2(f)=\int_0^{-\infty} \frac{|f(y)|}{a(y)^2}\Phi(y,-\infty)dy.
$$
\begin{theorem}[\cite{mijauru}]\label{miur}
For arbitrary $x\in\R$, the following statements hold.
\begin{itemize}
\item Let $\Phi(0,+\infty)< +\infty, \Phi(0,-\infty)=-\infty$.

If   $I_1(f)
<+\infty $, then $J^x_\infty(f)\in \R $ a.s.

If $I_1(f)
=\infty $ then $J^x_\infty(f)=\infty $ a.s.
\item Let $\Phi(0,+\infty)= +\infty, \Phi(0,-\infty)>-\infty$.

If  $I_2(f)
<+\infty $, then $J^x_\infty(f)\in \R $ a.s.

If $I_2(f)
=-\infty $ then $J^x_\infty(f)=\infty $ a.s.
\item Let $\Phi(0,+\infty)< +\infty, \Phi(0,-\infty)>-\infty$.

If   $I_1(f)
<+\infty $, then $J^x_\infty(f)\in \R $ a.s. on $A_+$.

If $I_1(f)
=\infty $ then $J^x_\infty(f)=\infty $ a.s. on $A_+$.

If $I_2(f)
<+\infty $, then $J^x_\infty(f)\in \R $ a.s. on $A_-$.

If $I_2(f)
=+\infty $ then $J^x_\infty(f)=\infty $ a.s. on  $A_-$.
\end{itemize}
\end{theorem}
In what follows consider the case when $\Phi(0,+\infty)< +\infty, \Phi(0,-\infty)=-\infty$, other cases being similar. Next result is a direct consequence of Corollary \ref{cor1}.
\begin{lemma}\label{lem1} Let $\Phi(0,+\infty)< +\infty, \Phi(0,-\infty)=-\infty$. Then for any $k\ge 1$
 $$ \ex{L_\infty^x(y)^k}= k!(2\Phi(y,+\infty))^k\; \text{for}\; x\leq y$$ and $$  \ex{L_\infty^x(y)^k}=2^k k!\Phi(y,+\infty)^{k-1}\big(\Phi(y,+\infty)-\Phi(y,x)\big)=2^k k!\Phi(y,+\infty)^{k-1}\varphi(y,x)\Phi(x,+\infty)$$ for $ x>y.$
\end{lemma}
\begin{example}
 Let $a=1$, $b=\mu>0$ with some constant $\mu$, so that $X$ is a Brownian motion with a constant positive drift. Furthermore, in this case $\Phi(y,+\infty)=1/{2\mu}$,  $\varphi(y,x)=\exp\{-2\mu(x-y)\}$, $\Phi(0,-\infty)=-\infty$. Therefore, the criterion for $J_\infty(f)$ to be finite is $\int_0^\infty |f(x)|dx<\infty$, which coincides with that of \cite{SalmYor}. As to the moments of local times,  in this case  $ \ex{L_\infty^x(y)}= {1}/{\mu}$ for $x\leq y$ and $\ex {L_\infty^x(y)}=\frac{1}{\mu}\exp\{-2\mu(x-y)\}$ for $x>y$.
\end{example}
Further we derive  conditions for  $\ex{J^x_\infty(f)}$ to be finite.
\begin{theorem}\label{theo3} Let $\Phi(0,+\infty)< +\infty, \Phi(0,-\infty)=-\infty$ and  $I_1(f)<+\infty $. Assume additionally that
$$\int_{-\infty}^x\frac{|f(u)|}{a(u)^2}\varphi(u,x)du<+\infty.$$  Then $$\ex{J^x_\infty(f)}=2\int_x^{+\infty} \frac{ f(u) }{a(u)^2}\Phi(u,+\infty)du+2\Phi(x, +\infty)\int_{-\infty}^x\frac{ f(u) }{a(u)^2}\varphi(u,x)du.$$
\end{theorem}
\begin{proof}
The statement immediately follows from Lemma \ref{lem1} and the Fubini theorem.
\end{proof}
\begin{remark} For a Brownian motion with a positive drift $\mu$, a sufficient condition for $\ex{J^x_\infty(f)}$ to be  finite is $$\int_x^{+\infty} | f(u)|  du+e^{-2\mu x}\int_{-\infty}^x |f(u)|e^{2\mu u}du<\infty;$$ and in that case the equality $$\ex{J^x_\infty(f)}=\frac{1}{\mu}\int_x^{+\infty}  f(u)  du+\frac{1}{\mu}e^{-2\mu x}\int_{-\infty}^x f(u) e^{2\mu u}du$$ holds. Obviously, the requirement  $\int_{\R}| f(u)| du<\infty$ is also sufficient, which is stated in \cite{SalmYor}.
\end{remark}
\begin{remark} Consider the class of locally integrable functions  $$f\colon \R \to \R,\; \text{there exists such}\;  x_0\in \R \;\text{such that}\, f(y)=0 \;\text{for any}\,  y<x_0.$$ It follows from Theorems \ref{miur} and \ref{theo3} that for any function $f$ from this class the integral $J_\infty^x(f)$ exists, and its expectation is finite.
\end{remark}

Now we continue with the moments of $J_\infty (f)$ of  higher order.
\begin{theorem}
Let $\Phi(0,+\infty)< +\infty, \Phi(0,-\infty)=-\infty$. The moments of higher order admit the following bound: for any $k>1$ \begin{equation}\begin{gathered}\label{eq1.19}
\left(\ex{\abs{J_\infty^x(f)}^k}\right)^{1/k}\leq 2(k!)^{1/k}\bigg(\int_x^{+\infty} \frac{|f(u)|}{a(u)^2}\Phi(u,+\infty)du\\{}+ \Phi(x, +\infty)^{{1}/{k}}\int_{-\infty}^x \frac{|f(u)|}{a(u)^2}\Phi(u,+\infty)^{1-{1}/{k}}\varphi(u,x)^{{1}/{k}}du.\bigg)\end{gathered}\end{equation}
\end{theorem}
\begin{proof} We use representation \eqref{eq1.17}
 and the generalized Minkowski inequality to get the following equalities and bounds:
\begin{equation}\begin{gathered}\label{equ1.20}
\left(\ex{\abs{J_\infty^x(f)}^k}\right)^{ {1}/{k}}=\left(\ex{\abs{\int_{\R}\frac{f(y)}{a(y)^2}L_{\infty}^x(y)dy}^k}\right)^{{1}/{k}}\leq 
\int_{\R}\frac{\abs{f(y)}}{a(y)^2}\left(\ex{L_{\infty}^x(y)^k}\right)^{{1}/{k}}dy.
\end{gathered}
\end{equation}
Now \eqref{eq1.19} follows immediately from \eqref{equ1.20} and Lemma \ref{lem1}.
\end{proof}
We conclude with the existence of potential and exponential moments. 
Some related results were obtained in \cite{salmyorkhosh}.
\begin{definition} The integral functional $J_\infty(f)$ has a bounded potential $P$ if $$P=\sup_{x\in \R}\ex{J_\infty^x(f)}<\infty.$$
\end{definition}
The following result is an immediate corollary of Theorem \ref{theo3}.
\begin{theorem}\label{theorem4} Let $\Phi(0,+\infty)< +\infty, \Phi(0,-\infty)=-\infty$ and  $$P_0=2\sup_{x\in \R} \bigg(\int_x^{+\infty} \frac{ |f(u)| }{a(u)^2}\Phi(u,+\infty)du+ \Phi(x, +\infty)\int_{-\infty}^x\frac{ |f(u)| }{a(u)^2}\varphi(u,x)du\bigg)<\infty.$$
Then the integral functional $J_\infty(f)$ has a bounded potential $P\leq P_0$.
\end{theorem}
\begin{theorem} Let $\Phi(0,+\infty)< +\infty, \Phi(0,-\infty)=-\infty$ and   $P_0<\infty.$   Then 
$$ \ex{\exp(\lambda J^x_\infty(f))}\leq \frac{1}{1-\lambda P_0}$$ for $\lambda< P_0$.\end{theorem}
\begin{proof} We apply the following result of Dellacherie and Meyer \cite{deme}, see also \cite[lemma 5.2]{SalmYor}. Let $A$ be a continuous adapted non-decreasing process starting at zero such that there exists a constant $C>0$ satisfying $\ex{A_\infty-A_t\mid \F_t}\leq C$ for any $t\geq 0$. Then $$ \ex{\exp(\lambda A_\infty)}\leq \frac{1}{1-\lambda C}$$ for $\lambda< C^{-1}$.

Set $A_t=\int_0^t |f(X^x_s)|ds$. Then it follows from Markov property of $X$ and Theorems \ref{theo3} and \ref{theorem4} that  $\ex{A_\infty-A_t\mid \F_t}\leq P_0$ for any $t\geq 0$ whence the proof follows.
\end{proof}

\bibliographystyle{elsarticle-num}
\bibliography{mishshevper}

\end{document}